\documentclass[a4paper,11pt]{amsart}
\usepackage[letterpaper, margin=1.2in]{geometry}
\usepackage{latexsym}
\usepackage{amssymb}
\usepackage{amsmath}
\usepackage{amsfonts}
\usepackage{color, soul}
\usepackage{tikz}

\newtheorem{thm}{Theorem}[section]

\newtheorem{lem}[thm]{Lemma}

\newtheorem{cor}[thm]{Corollary}

\theoremstyle{definition}
\newtheorem*{dfn}{Definition}

\newtheorem{example}{Example}

\newcommand{\R}{\mathbb R}
\newcommand{\N}{\mathbb N}
\newcommand{\Z}{\mathbb Z}

\newcommand{\p}{\mathfrak{p}}

\def\al{\alpha}
\def\la{\lambda}
\def\s{\sigma}
\def\d{\delta}
\def\t{\theta}
\def\om{\omega}

\def\Ga{\Gamma}
\def\ol{\overline}

\newcounter{cs}
\stepcounter{cs}
\newcommand{\casos}{\begin{itemize}}
\newcommand{\fcasos}{\end{itemize}\setcounter{cs}{1}}

\newfont{\tit}{cmr12 scaled \magstep3}

\begin{document}

\title[]{A Dedekind's Criterion over Valued Fields}

\author{Lhoussain El Fadil}
\address[L. El Fadil]{Department of Mathematics, Faculty of Sciences Dhar-Mahraz, University of Sidi Mohamed Ben Abdellah, B.P. 1796, Fes, Morocco} \email{lhouelfadil2@gmail.com}

\author{Mhammed Boulagouaz}
\address[M. Boulagouaz]{Department of Mathematics, Faculty of Sciences and Technologies, University of Sidi Mohamed Ben Abdellah, B.P. 2202, Fes, Morocco} \email{boulag@rocketmail.com}

\author{Abdulaziz Deajim}
\address[A. Deajim]{Department of Mathematics, King Khalid University, P.O. Box 9004, Abha, Saudi Arabia} \email{deajim@kku.edu, deajim@gmail.com}

\keywords{Dedekind's Criterion, valued field, extensions of a valuation, integral closure}
\subjclass[2010]{12J10, 13A18, 13B22}
\date{\today}

\begin {abstract}
Let $(K,\nu)$ be an arbitrary-rank valued field, $R_\nu$ its valuation ring, $K(\al)/K$ a separable finite field extension generated over $K$ by a root of a monic irreducible polynomial $f\in R_\nu[X]$. We give necessary and sufficient conditions for $R_\nu[\al]$ to be integrally closed. We further characterize the integral closedness of $R_\nu[\al]$ based on information about the valuations on $K(\al)$ extending $\nu$. Our results enhance and generalize some existing results in the relevant literature. Some applications and examples are also given.
\end {abstract}
\maketitle

\section{{\bf Introduction}}\label{intro}

For any valued field $(K, \nu)$, we denote by $\ol{K}$ an algebraic closure of $K$, $R_\nu$ the valuation ring of $\nu$, $M_\nu$ the maximal ideal of $R_\nu$, $k_\nu=R_\nu/M_\nu$ the residue field of $\nu$, and $\Ga_\nu$ the (totally ordered abelian) value group of $\nu$. We denote by $\Ga_\nu^+$ the set of elements $g\in \Ga_\nu$ such that $g>0$. We denote a minimum element of $\Ga_\nu^+$, if any, by $\min(\Ga_\nu^+)$. We also denote by $\nu^{\mbox{\tiny G}}$ the Gaussian extension of $\nu$ to the field $K(X)$ of rational functions; that is, for $f(X)=\sum_{i=0}^m a_i X^i\in K[X]$, set $\nu^{\mbox{\tiny G}}(f)=\min\{\nu(a_0), \dots, \nu(a_m)\}$ and extend to $K(X)$ as  $\nu^{\mbox{\tiny G}}(f/g)=\nu^{\mbox{\tiny G}}(f)-\nu^{\mbox{\tiny G}}(g)$ for $f, g\in K[X]$ and $g\neq 0$.

Let $(K,\nu)$ be a valued field of arbitrary rank, $f\in R_\nu[X]$ a monic irreducible separable polynomial, $\al\in \ol{K}$ a root of $f$, $L=K(\al)$ the simple field extension over $K$ generated by $\al$, and $S$ the integral closure of $R_\nu$ in $L$. Assume that $\ol{f}=\prod_{i=0}^s \ol{\phi_i}^{\, l_i}$ is the monic irreducible factorization of $\ol{f}$ over $k_\nu$, and $\phi_i\in R_\nu[X]$ is a monic lifting of $\ol{\phi_i}$ for $i=1, \dots, s$. For the sake of brevity, we shall refer to these notations and assumptions as {\it\textbf{Assump's}}.

Under {\it\textbf{Assump's}}, if $R_\nu$ is a discrete valuation ring and $M_\nu$ does not divide the index ideal $[S:R_\nu[\al]]$, then a well-known theorem of Kummer (see \cite[Proposition 8.3]{Neu} for instance) gives the factorization of the ideal $M_\nu \,S$; namely, $M_\nu\,S=\prod_{i=1}^s \p_i^{l_i}$, where $\displaystyle{\p_i=M_\nu\,S+\phi_i(\al)\,S}$ with residue degree equal to $\mbox{deg}(\phi_i)$. R. Dedekind (\cite{Ded}) gave a criterion for the divisibility of $\displaystyle{[S: R[\al]]}$ by $M_\nu$, which was also extended in \cite{KK2}. For an arbitrary valuation $\nu$ in general, Y. Ershov (in \cite{Er}) introduced a nice generalized version of Dedekind's Criterion. Namely, he showed that if we write $f$ in the form $f=\prod_{i=1}^s \phi_i^{l_i}+ \pi T$ for some $\pi\in M_\nu$ and $\displaystyle{T\in (R_\nu-M_\nu)[X]}$, then $R_\nu[\al]$ is integrally closed (i.e. $R_\nu[\al]=S$) if and only if either $l_i =1$ for all $i=1, \dots, s$ or, else, $\nu(\pi) =\min(\Ga_\nu^+)$ and $\ol{\phi_i}$ does not divide $\ol{T}$ for all those $i=1, \dots, s$ with $l_i \geq 2$. S. Khanduja and M. Kumar gave a different elegant proof of Ershov's result in \cite[Theorem 1.1]{KK1}. It should be noted, however, that the proof given in \cite{KK1} is to be taken within the separability context despite not stating this, as the proof relies on \cite[17.17]{End} which explicitly requires the separability of $f$.

on does not explicitly state the assumption that $f$ is separable, both references crucially use, which in turn assumes the separability of $f$

Assuming {\it\textbf{Assump's}}, we give in Theorem \ref{main} another version of Dedekind's Criterion where we utilize the Euclidean division of $f$ by $\phi_i$ for every $i=1, \dots, s$ with $l_i\geq 2$ in such a way that computationally enhances \cite[Theorem 1.1]{KK1} and further improves \cite[Theorem 4.1]{KK1} as our result does not require $K$ to be Henselian. In Theorem \ref{main 2}, we give a complete characterization of the integral closedness of $R_\nu[\al]$ based on the valuations of $L$ extending $\nu$ and their values at $\phi_i(\al)$ for $i=1, \dots, s$ with $l_i\geq 2$. In this case, we further compute the ramification indices and residue degrees of all the valuations of $L$ extending $\nu$ (Corollary \ref{ram-res}). 
Some further applications and examples are given in Section \ref{applications}


\section{{\bf The Main Results}}

Keeping the notations of {\it\textbf{Assump's}}, denote by $(K^h, \nu^h)$ the Henselization of $(K, \nu)$ and by $\ol{\nu^h}$ the unique extension of $\nu^h$ to the algebraic closure $\ol{K^h}$ of $K^h$. 

We begin this section with the following important, well-known result, which we present without proof (see for instance \cite[17.17]{End}). The result asserts a one-to-one correspondence between the valuations on $L$ extending $\nu$ and the irreducible factors of $f$ over $K^h$.

\begin{lem}\label{Endler}
Keep the notation and assumptions as in \textbf{Assump's}. Let $f=\prod_{j=1}^t f_j$ be the factorization of $f$ into a product of distinct monic irreducible polynomials over $K^h$. Then there are exactly $t$ extensions $\om_1, \dots, \om_t$ of $\nu$ to $L$. Morevoer, if $\al_j$ is any root of $f_j$ in $\ol{K^h}$ for \linebreak $j\in \{1, \dots, t\}$, then the valuation $\om_j$ corresponding to $f_j$ is precisely the valuation on $L$ \linebreak satisfying: $\om_j(h(\al))=\ol{\nu^h}(h(\al_j))$ for any $h\in K[X]$.
\end{lem}

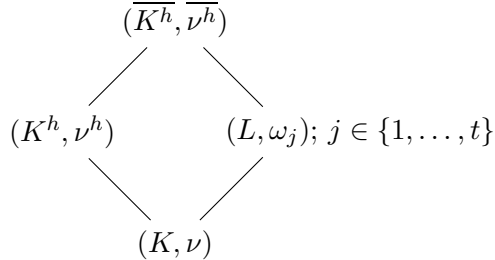
\begin{center}
\begin{tikzpicture}[node distance = 1.5cm, auto]
      \node (K) {$(K,\nu)$};
      \node (Kh) [above of=K, left of=K] {$(K^h, \nu^h)$};
      \node (L) [above of=K, right of=K] {$\qquad \qquad\quad(L, \om_j); \, j\in\{1, \dots ,t\}$};
      \node (KK) [above of=K, node distance = 3cm] {$(\ol{K^h}, \ol{\nu^h})$};
      \draw[-] (K) to node {} (Kh);
      \draw[-] (K) to node {} (L);
      \draw[-] (Kh) to node {} (KK);
      \draw[-] (L) to node {} (KK);
      \end{tikzpicture}
\end{center}

The following result is a generalization of \cite[Lemm 2.1]{DE} to arbitrary-rank valuations.

\begin{lem}\label{lemma} Keep the notation and assumptions as in \textbf{Assump's} and Lemma \ref{Endler}.
\begin{itemize}
\item[(i)] For every $i=1, \dots, s$, $\om_j(\phi_i(\al))>0$ for some $j=1, \dots, t$.
\item[(ii)] For every $j=1, \dots, t$ and every nonzero $p\in R_\nu[X]$, $\om_j(p(\al))\geq \nu^{\mbox{\tiny G}}(p(X))$.
\item[(iii)] For every $j=1, \dots, t$, there exists a unique $i=1, \dots, s$ such that $\om_j(\phi_i(\al))>0$. Moreover, $\om_j(\phi_k(\al))=0$ for all $k\neq i$, $k=1, \dots, s$.
\item[(iv)] Equality holds in (ii) if and only if $\overline{\phi_i}$ does not divide $\overline{(p/a)}$ for the unique index $i$ associated to $\om_j$ in (iii), where $a$ is any coefficient of $p$ of minimum $\nu$-valuation.
\end{itemize}
\end{lem}

\begin{proof}\hfill
\begin{itemize}
\item[(i)] 
    Since $k_{\nu^h}=k_\nu$, $\prod_{i=1}^s \ol{\phi_i}^{\,l_i} =\prod_{j=1}^t \ol{f_j}$. So, for a fixed $i=1, \dots, s$, there is some $j=1, \dots, t$ such that $\ol{\phi_i}$ divides $\ol{f_j}$. Since $f_j$ is irreducible, it follows from Hensel's Lemma that $\ol{f_j}=\ol{\phi_i}^{\,u_i}$ for some $1\leq u_i \leq l_i$. Let $\al_j\in \ol{K^h}$ be a root of $f_j$. As $f_j(\al_j)=0$, we have $\ol{\phi_i(\al_j)}^{\,u_j} = \ol{f_j(\al_j)} =\ol{0}$ modulo $M_{\ol{\nu^h}}$. Thus, $\phi_i(\al_j)^{u_i}\in M_{\ol{\nu^h}}$ and so $\phi_i(\al_j)\in M_{\ol{\nu^h}}$. Now, by Lemma \ref{Endler}, $\om_j(\phi_i(\al))=\ol{\nu^h}(\phi_i(\al_j))>0$ as desired.
\item[(ii)] Set $p_1=p/a$, where $a$ is a coefficient of $p$ of least $\nu$-valuation. As $\nu^{\mbox{\tiny G}}(p_1)=0$, $p_1 \in R_\nu[X]$. Since $S = \bigcap_{j=1}^t R_{\om_j}$ (see \cite[Corollary 3.1.4]{Eng}), it follows that, for every $j=1, \dots, t$, we have $p_1(\al)\in R_\nu[\al]\subseteq S\subseteq R_{\om_j}$ and $$\qquad \quad \om_j(p(\al))=\om_j(a)+\om_j(p_1(\al))=\nu(a)+\om_j(p_1(\al))=\nu^{\mbox{\tiny G}}(p(X))+\om_j(p_1(\al)) \geq \nu^{\mbox{\tiny G}}(p(X))$$ as claimed.
\item[(iii)] Fix a $j=1, \dots, t$. Since $\prod_{i=1}^s \phi_i(\al)^{l_i} \equiv f(\al) \equiv 0 \; (\mbox{mod} \; M_{\om_j})$, $\om_j(\prod_{i=1}^s \phi_i(\al)^{l_i}) >0$. So, $\om_j(\phi_i(\al)) >0$ (and so $\phi_i(\al)\in M_{\om_j}$) for some $i=1, \dots, s$. For $k=1, \dots, s$ with $k\neq i$, as $\ol{\phi_i}$ and $\ol{\phi_k}$ are coprime modulo $M_\nu$, we let $s_k, t_k \in R_\nu[X]$ be such that $\overline{s_k} \, \overline{\phi_i} + \overline{t_k} \, \overline{\phi}_k \equiv 1 \;(\mbox{mod} \;M_\nu)$. Then, $s_k(\al)\phi_i(\al) + t_k(\al)\phi_k(\al) = 1 +h(\al)$ for some $h \in M_\nu[X]$. As $\nu^{\mbox{\tiny G}}(h)>0$, it follows from part (ii) that $\om_j(h(\al))>0$ and so $h(\al) \in M_{\om_j}$. Since $\phi_i(\al)\in M_{\om_j}$ and $s_k(\al)\in R_\nu[\al]\subseteq S \subseteq R_{\om_j}$, $s_k(\al) \phi_i(\al))\in M_{\om_j}$. Thus, $\displaystyle{t_k(\al)\phi_k(\al)\in R_{\om_j} - M_{\om_j}}$. So, $\om_j(t_k(\al)\phi_k(\al))=0$ and thus $\om_j(\phi_k(\al)) =0$. The uniqueness of $i$ such that $\om_j(\phi_i(\al))>0$ follows.
\item[(iv)] Define the map $\textcolor[rgb]{0.00,0.07,1.00}{\psi_j}: k_\nu[X] \to R_{\om_j}/M_{\om_j}$ by $\overline{p}(X) \mapsto p(\al)+ M_{\om_j}$. Since $M_\nu \subseteq M_{\om_j}$, $\psi_j$ is a well-defined ring homomorphism. As $\om_j(p(\al))=\nu^{\mbox{\tiny G}}(p(X)) +\om_j(p_1(\al))$ (see part (ii)), it follows that $\om_j(p(\al))=\nu^{\mbox{\tiny G}}(p(X))$ if and only if $\om_j(p_1(\al))=0$, if and only if $p_1(\al)\in R_{\om_j} - M_{\om_j}$, if and only if $\overline{p_1}(X) \not\in \mbox{ker}\,\psi_j$. By part (iii), let $\phi_i$ be such that $\om_j(\phi_i(\al))>0$. Then, $\phi_i(\al)\in M_{\om_j}$ and so $\overline{\phi_i}\in \mbox{ker}\,\psi_j$. Since $\mbox{ker}\,\psi_j$ is a principal ideal of $k_\nu[X]$ and $\overline{\phi_i}$ is irreducible over $k_\nu$, $\mbox{ker}\,\psi_j$ is generated by $\overline{\phi_i}$. It now follows that $\om_j(p(\al))=\nu^{\mbox{\tiny G}}(p)$ if and only if $\overline{\phi_i}$ does not divide $\overline{p_1}$.
\end{itemize}
\end{proof}



Keeping the notation of \textbf{Assump's}, in what follows we let $q_i, r_i\in R_\nu[X]$ be, respectively, the quotient and the remainder upon the Euclidean division of $f$ by $\phi_i$, for $i=1, \dots, s$.

In \cite[lemma 2.1 (b)]{KK1}, it was shown that $\Ga_\nu^+$ contains a smallest element in case $R_\nu[\al]$ is integrally closed and $l_i\geq 2$ for some $i=1, \dots, s$. Below, we prove this fact differently with some thing extra.

\begin{lem}\label{inf}
Keep the notation and assumptions as in Lemma \ref{lemma}. If $R_\nu[\al]$ is integrally closed and $I=\{i\,|\, l_i\geq 2,\, i=1, \dots, s \}$ is not empty, then $\Ga_\nu^+$ has a minimum element with $\min(\Ga_\nu^+)=\nu^{\mbox{\tiny G}}(r_i)$ for every $i\in I$.
\end{lem}

\begin{proof}
For $i\in I$, let $q_i^*, r_i^*\in R_\nu[X]$ be, respectively, the quotient and remainder upon the Euclidean division of $q_i$ by $\phi_i$. Since $\ol{\phi_i}$ divides both $\ol{f}$ and $\ol{q_i}\ol{\phi_i}$, $\ol{\phi_i}$ divides $\ol{r_i}$. But, as $\phi_i$ is monic, $\mbox{deg}(\ol{\phi_i})=\mbox{deg}(\phi_i)> \mbox{deg}(r_i) \geq \mbox{deg}(\ol{r_i})$. This implies that $\ol{r_i}$ is zero and so $\nu^{\mbox{\tiny G}}(r_i)>0$. Thus, $\nu^{\mbox{\tiny G}}(r_i)\in \Ga_\nu^+$. Now as $\ol{f}=\ol{q_i}\ol{\phi_i}$ and $\ol{\phi_i}^2$ divides $\ol{f}$, $\ol{\phi_i}$ must divide $\ol{q_i}$. Applying a similar argument to the expression $\ol{q_i}=\ol{q_i^*}\,\ol{\phi_i}+\ol{r_i^*}$, we get that $\ol{r_i^*}$ is zero. So, $\nu^{\mbox{\tiny G}}(r_i^*)>0$ and, thus, $\nu^{\mbox{\tiny G}}(r_i^*)\in \Ga_\nu^+$. To the contrary, suppose that $\tau_i\in \Ga_\nu^+$ is such that $\tau_i < \nu^{\mbox{\tiny G}}(r_i)$, and set $\d_i=\min \{\tau_i, \nu^{\mbox{\tiny G}}(r_i)-\tau_i, \nu^{\mbox{\tiny G}}(r_i^*)\}$. As $\d_i\in \Ga_\nu^+$, let $d_i\in R_\nu$ be such that $\nu(d_i)=\d_i$ and set $\t_i=q_i(\al)/d_i$. Let $\om$ be a valuation of $L$ extending $\nu$. We show that $\t_i\in R_\om$ and, since $\om$ is arbitrary, it would follow that $\t_i \in S$ (\cite[Corollary 3.1.4]{Eng}). As $f(\al)=0$, $\t_i=-r_i(\al)/(d_i \phi_i(\al))$. By Lemma \ref{lemma}, let $j\in \{1, \dots, s\}$ be the unique index such that $\om(\phi_j(\al))>0$ and $\om(\phi_k(\al))=0$ for all $k\in \{1, \dots, s\}-\{j\}$. If $i\neq j$, then $\om(\phi_i(\al))=0$ and $$\om(\t_i)=\om(r_i(\al))-\om(d_i)=\om(r_i(\al))-\nu(d_i)\geq \nu^{\mbox{\tiny G}}(r_i)-\d_i>\d_i-\d_i=0,$$ and so $\t_i\in R_\om$ in this case. Assume, on the other hand, that $i=j$. If $\om(\phi_i(\al))> \d_i$, then as $q_i^*$ is monic and $\om(q_i^*(\al))\geq \nu^{\mbox{\tiny G}}(q_i^*)=0$ (Lemma \ref{lemma}), we have
$$\om(q_i(\al))\geq \min\{\om(q_i^*(\al)+\phi_i(\al)), \om(r_i^*(\al))\} \geq \min\{\om(\phi_i(\al)), \nu^{\mbox{\tiny G}}(r_i^*)\} \geq \d_i.$$
So, $\om(\t_i)=\om(q_i(\al))-\om(d_i)>\d_i -\d_i=0$, which implies that $\t_i\in R_\om$ in this case too. If, on the other hand, $\om(\phi_i(\al))\leq \d_i$, then
$$\om(\t_i)=\om(r_i(\al))-\om(d_i)-\om(\phi_i(\al))\geq \nu^{\mbox{\tiny G}}(r_i)-\d_i -\d_i\geq \nu^{\mbox{\tiny G}}(r_i)-\tau_i-\d_i\geq \d_i-\d_i=0.$$
So, $\t_i\in R_\om$ in this case as well. It now follows from the above argument that $\t_i\in S$. But, as $q_i$ is monic and $1/d_i\not\in R_\nu$, it is clear that $\t_i \not\in R_\nu[\al]$, contradicting the assumption that $R_\nu[\al]$ is integrally closed. Hence, $\nu^{\mbox{\tiny G}}(r_i)$ is the minimal element of $\Ga_\nu^+$ as claimed.
\end{proof}

\begin{lem}\label{sigma}
Keep the notation and assumptions as in Lemma \ref{lemma}. If $\min(\Ga_\nu^+)=\s$, then for any $i\in\{1, \dots, s\}$ with $\nu^{\mbox{\tiny G}}(r_i)=\s$, and for any valuation $\om$ of $L$ extending $\nu$ such that $\om(\phi_i(\al))>0$, we have $\om(\phi_i(\al))=\s/l_i$.
\end{lem}

\begin{proof}
Let $i\in \{1, \dots ,s\}$ and $\om$ be a valuation of $L$ extending $\nu$ such that $\om(\phi_i(\al))>0$. Write $f$ in the form $f=m_i\phi_i^{l_i}+ n_i \phi_i+ r_i$, with $m_i, n_i\in R_\nu[X]$, $\nu^{\mbox{\tiny G}}(m_i)=0$, $\ol{\phi_i}$ does not divide $\ol{m_i}$, $\nu^{\mbox{\tiny G}}(n_i)>0$, and $\mbox{deg}(r_i)<\mbox{deg}(\phi_i)$. Notice that if $l_i=1$, then $m_i=q_i$ and $n_i=0$. By Lemma \ref{lemma}, $\om(n_i(\al))\geq \nu^{\mbox{\tiny G}}(n_i)\geq \s$, $\om(m_i(\al))=\nu^{\mbox{\tiny G}}(m_i)=0$, and $\om(r_i(\al))=\nu^{\mbox{\tiny G}}(r_i)=\s$ as $\ol{\phi_i}$ divides neither $\ol{m_i}$ nor $\ol{r_i}$. We then have
$$l_i\om(\phi_i(\al))=\om(m_i(\al)\phi_i^{l_i}(\al))=\om(n_i(\al)\phi_i(\al)+r_i(\al))=\om(r_i(\al))=\nu^{\mbox{\tiny G}}(r_i)=\s$$ as claimed.
\end{proof}


Now we get to our first main result, which computationally enhances \cite[Theorem 1.1]{KK1} and, further, improves \cite[Theorem 4.1]{KK1} in the sense that $K$ is not assumed to be Henselian.

\begin{thm}\label{main}
Keep the notation and assumptions as in Lemma \ref{lemma}.
\begin{itemize}
\item[(i)] If $l_i=1$ for all $i=1, \cdots, s$, then $R_\nu[\al]$ is integrally closed.

\item[(ii)] If $I=\{i\,|\, l_i\geq 2,\, i=1, \dots, s \}$ is not empty, then $R_\nu[\al]$ is integrally closed if and only if $\nu^{\mbox{\tiny G}}(r_i)=\min(\Ga_\nu^+)$ for every $i\in I$.
\end{itemize}
\end{thm}

\begin{proof}$\hfill$
\begin{itemize}
\item[(i)] Assume that $l_i=1$ for all $i=1, \dots, s$. An arbitrary element of $S$ is of the form \linebreak $\t=h(\al)/b$ for some $b\in R_\nu$ and $h\in R_\nu[X]$, with $\nu^{\mbox{\tiny G}}(h)=0$ and $\displaystyle{\mbox{deg}(h)<\mbox{deg}(f)}$. Since $f$ is monic, $\mbox{deg}(\ol{h})\leq \mbox{deg}(h)<\mbox{deg}(f)=\mbox{deg}(\ol{f})$. As $l_i=1$ for all \linebreak $i=1,\dots, s$, there is some $i=1, \dots, s$ such that $\ol{\phi_i}$ does not divide $\ol{h}$. For such a fixed $i$, let $\om$ be a valuation of $L$ extending $\nu$ such that $\om(\phi_i(\al))>0$, according to Lemma \ref{lemma}. By Lemma \ref{lemma} again, $\om(h(\al))=\nu^{\mbox{\tiny G}}(h)=0$. If $\nu(b)>0$, then $\displaystyle{\om(\t)=\om(h(\al))-\om(b)=0-\nu(b)<0}$. Thus, $\t \not\in S$, which is a contradiction. Hence, $\nu(b)=0$, which implies that $\t \in R_\nu[\al]$. This shows that $S=R_\nu[\al]$ and, hence, $R_\nu[\al]$ is integrally closed.
\item[(ii)] Assume that $I\neq \varnothing$. If $R_\nu[\al]$ is integrally closed, then it follows from Lemma \ref{inf} that $\nu^{\mbox{\tiny G}}(r_i)$ is the minimum element of $\Ga_\nu^+$ for every $i\in I$, as claimed.

    Conversely, set $\min(\Ga_\nu^+)=\s$ and let $\pi \in R_\nu$ be such that $\nu(\pi)=\s$. Assume that, for every $i\in I$, $\nu^{\mbox{\tiny G}}(r_i)=\s$. We aim at proving that $R_\nu[\al]$ is integrally closed. By making an appropriate choice of a lifting of $\ol{\phi_i}$, we begin by showing that we can also assume that $\nu^{\mbox{\tiny G}}(r_i)=\s$ for $i\not\in I$. Let $i\not\in I$, and assume that $\nu^{\mbox{\tiny G}}(r_i)>\s$. If $\d \in \Ga_\nu^+$ with $\s<\d<2\s$, then $\d-\s\in \Ga_\nu^+$ with $\d-\s<2\s-\s=\s$ contradicting the minimality of $\s$. So there is no element of $\Ga_\nu^+$ lying strictly between $\s$ and $2\s$. So, $\nu^{\mbox{\tiny G}}(r_i)\geq 2\s$. Let $q_i^*, r_i^*\in R_\nu[X]$ be, respectively, the quotient and remainder upon the Euclidean division of $q_i$ by $\phi_i$. Set $\phi_i^{**}=\phi_i +\pi$, $q_i^{**}=q_i-\pi q_i^*$, and $r_i^{**} =r_i-\pi r_i^*+\pi^2 q_i^*$. Then we have
\begin{align*}
q_i^{**}\phi_i^{**}+r_i^{**} &=(q_i-\pi q_i^*)(\phi_i+\pi) +r_i-\pi r_i^* +\pi^2 q_i^*\\
                   &=q_i\phi_i+r_i\\
                   &= f.
\end{align*}
It can be easily checked that $q_i^{**}$ and $r_i^{**}$ are, respectively, the quotient and remainder upon the Euclidean division of $f$ by $\phi_i^{**}$. Since $l_i=1$, $\ol{r_i^*}$ is nonzero and so \linebreak $\nu^{\mbox{\tiny G}}(\pi r_i^*)=\nu(\pi)=\s$. As $\nu^{\mbox{\tiny G}}(r_i) \geq 2\s$ and $\nu^{\mbox{\tiny G}}(\pi^2 q_i^*) \geq \nu(\pi^2)= 2\s$, it follows that $\nu^{\mbox{\tiny G}}(r_i^{**})=\nu^{\mbox{\tiny G}}(\pi r_i^*)=\s$. So, replacing $\phi_i$ by $\phi_i+\pi$, we can assume that $\nu^{\mbox{\tiny G}}(r_i)=\s$. We thus assume in the remainder of the proof that $\nu^{\mbox{\tiny G}}(r_i)=\s$ for all $i=1, \dots, s$. We finally get to proving that $R_\nu[\al]$ is integrally closed. Assume, to the contrary, that there exists some $\t \in S-R_\nu[\al]$. Then $\t$ can be written in the form $\t=g(\al)/b$ for some $b\in R_\nu$ and $g\in R_\nu[X]$ with $\nu(b)\geq \s$, $\nu^{\mbox{\tiny G}}(g)=0$, and $\mbox{deg}(g)<\mbox{deg}(f)$. For each $i=1, \dots, s$, let $m_i\geq 0$ be the highest power of $\ol{\phi_i}$ that divides $\ol{g}$. Since $\mbox{deg}(g)< \mbox{deg}(f)$, there must exist some $i=1, \dots, s$ such that $m_i\leq l_i-1$. For such an $i$, apply the Euclidean division of $g$ by $\phi_i^{m_i}$ to get $g=S_i\phi_i^{m_i}+T_i$, where $S_i, T_i\in R_\nu[X]$, $\ol{\phi_i}$ does not divide $\ol{S_i}$, and $\nu^{\mbox{\tiny G}}(T_i)\geq \s$. By Lemma \ref{lemma}, let $\om$ be a valuation of $L$ extending $\nu$ such that $\om(\phi_i(\al))>0$. Since $\ol{\phi_i}$ does not divide $\ol{S_i}$ and $S_i$ is monic, it follows from Lemma \ref{lemma} that $\om(S_i(\al))=\nu^{\mbox{\tiny G}}(S_i)=0$. Using Lemma \ref{sigma}, we then have $$\om(S_i(\al)\phi_i(\al)^{m_i})=m_i\om(\phi_i(\al))=m_i \s/l_i.$$ Since $\om(T_i(\al))\geq \nu^{\mbox{\tiny G}}(T_i)\geq \s$ (by Lemma \ref{lemma}), it follows that $$\om(g(\al))=\min\{\om(S_i(\al)\phi_i(\al)^{m_i}), \om(T_i(\al))\}=\min\{m_i \s/l_i, \s\}=m_i \s/l_i <\s.$$ Thus, $\om(\t)=\om(g(\al))-\om(b)=\om(g(\al))-\nu(b)<\s-\s=0$. Hence, $\t\not\in R_\om$ and, hence, $\t\not\in S$. This contradiction leads to the conclusion that $S=R_\nu[\al]$, as desired.
\end{itemize}
\end{proof}

The following corollary is immediate.

\begin{cor}
Keep the assumptions of Theorem \ref{main}. If $\mbox{$\Ga_\nu^+$}$ does not have a minimum element, then $R_\nu[\al]$ is integrally closed if and only if $l_i=1$ for every $i=1, \dots, s$.
\end{cor}

\begin{cor}
Keep the assumptions of Theorem \ref{main}. If $\Ga_\nu^+$ has a minimum element $\s$ and $I=\{i\,|\, l_i\geq 2,\, i=1, \dots, s \}$ is not empty, then $R_\nu[\al]$ is integrally closed if and only if $\ol{\phi_i}$ does not divide $\ol{M}$ for every $i\in I$, where $M=\cfrac{f-\prod_{i=1}^s \phi_i^{l_i}}{\pi}$ for any $\pi\in R_\nu$ with $\nu(\pi)=\s$.
\end{cor}

\begin{proof}
Let $i\in I$. Since $\ol{r_i}=\ol{f}-\ol{q_i}\;\ol{\phi_i}$ and $\ol{\phi_i}$ divides $\ol{f}$, $\ol{r_i}$ is divisible by $\ol{\phi_i}$. But as $\mbox{deg}(\ol{r_i})\leq \mbox{deg}(r_i) < \mbox{deg}(\phi_i)=\mbox{deg}(\ol{\phi_i})$, $\ol{r_i}$ must be zero. Thus, $\ol{q_i}=\ol{\phi_i^{l_i-1}} \prod_{j=1, j\neq i}^s \ol{\phi_j^{l_j}}$. Let $H_i\in R_\nu[X]$ be such that $q_i=\phi_i^{l_i-1} \prod_{j=1, j\neq i}^s \phi_j^{l_j}+\pi H_i$ with $\pi\in R_\nu$ such that $\nu(\pi)=\s$. Then,
$f=(\phi_i^{l_i-1} \prod_{j=1, j\neq i}^s \phi_j^{l_j}+\pi H_i)\phi_i +r_i$.
Set $M=\cfrac{f-\prod_{j=1}^s \phi_j^{l_j}}{\pi} \in R_\nu[X]$. Then,
\begin{align*}
M&=\cfrac{(\phi_i^{l_i-1} \prod_{j=1, j\neq i}^s \phi_j^{l_j}+\pi H_i)\phi_i +r_i-\prod_{j=1}^s \phi_j^{l_j}}{\pi}\\
&= H_i \phi_i + \cfrac{r_i}{\pi}.
\end{align*}
Since $M, H_i\phi_i \in R_\nu[X]$, we must have $\cfrac{r_i}{\pi} \in R_\nu[X]$ and so $\nu^G(\cfrac{r_i}{\pi})\geq \s$. Clearly, $\ol{\phi_i}$ divides $\ol{M}$ if and only if $\ol{\phi}$ divides $\cfrac{\ol{r_i}}{\ol{\pi}}$. As $\mbox{deg}(\cfrac{\ol{r_i}}{\ol{\pi}})=\mbox{deg}(\ol{r_i})< \mbox{deg}(\ol{\phi_i})$ (see above), we conclude that $\ol{\phi_i}$ divides $\ol{M}$ if and only if $\cfrac{\ol{r_i}}{\ol{\pi}}$ is zero, that is $\nu^G(r_i)>\s$. Contrapositively, $\ol{\phi_i}$ does not divide $\ol{M}$ if and only if $\nu^G(r_i)=\s$.
\end{proof}

Our second main result gives a new characterization of the integral closedness of $R_\nu[\al]$ based on characterization of the extensions of $\nu$ to $L$. The following definition and lemma are needed.

\begin{dfn}
We say that a monic polynomial $g\in R_\nu[X]$ is {\it $\nu$-Eisenstein} if there exists a monic polynomial $\psi\in R_\nu[X]$ such that $\ol{\psi}$ is irreducible, $\ol{g}$ is a positive power of $\ol{\psi}$, and $\nu^{\mbox{\tiny G}}(r)=\mbox{min}(\Ga_\nu^+)$, where $r\in R_\nu[X]$ is the remainder upon the Euclidean division of $g$ by $\psi$.
\end{dfn}

\begin{lem}\label{Eis 1}
Keep the assumptions of Theorem \ref{main}. Let $\min(\Ga_\nu^+)=\s$ and $g \in R_\nu[X]$ monic. If $g$ is $\nu$-Eisenstein, then $g$ is irreducible over $K$.
\end{lem}

\begin{proof}
Let $\psi \in R_\nu[X]$ be monic such that $\ol{\psi}$ is irreducible, $\ol{g} =\ol{\psi}^l$, and $\nu^{\mbox{\tiny G}}(r)=\s$, where $r\in R_\nu[X]$ is the remainder upon the Euclidean division of $g$ by $\psi$. To the contrary, suppose that $g=h_1 h_2$ for some non-constant and monic $h_1, h_2\in R_\nu[X]$. Then we have $\ol{h_1}=\ol{\psi}^{l_1}$ and $\ol{h_2}=\ol{\psi}^{l_2}$ for some positive $l_1, l_2$ with $l_1+l_2=l$. Assume that the Euclidean division of each of $g, h_1, h_2$ by $\psi$ yield
$$g=q\psi+r,\;\; h_1=q_1 \psi + r_1, \;\; h_2=Q_2\psi+ r_2.$$
It is clear that $r$ is the remainder upon the Euclidean division of the product $r_1 r_2$ by $\psi$. Since both $\ol{h_1}$ and $\ol{h_2}$ are positive powers of $\ol{\psi}$, both of  $\ol{r_1}$ and $\ol{r_2}$ must be zero. So, $\nu^{\mbox{\tiny G}}(r_1)\geq \s$ and $\nu^{\mbox{\tiny G}}(r_2)\geq \s$. Thus, $\nu^{\mbox{\tiny G}}(r)\geq 2\s> \s$ (as $\s>0$), which is a contradiction. Hence, $g$ is irreducible over $R_\nu$ and, consequently, irreducible over $K$ (by Gauss's Lemma, as $R_\nu$ is integrally closed).
\end{proof}

\begin{thm}\label{main 2}
Keep the assumptions of Theorem \ref{main}. The following statements are equivalent:
\begin{itemize}
\item[(i)] $R_\nu[\al]$ is integrally closed.
\item[(ii)] $\nu$ has exactly $s$ distinct extensions $\om_1, \dots, \om_s$ to $L$, and if $I=\{i\,|\, l_i\geq 2,\, i=1, \dots, s \}$ is not empty, then $l_i\om_i(\phi_i(\al))$ is the minimum element of $\Ga_\nu^+$ for every $i\in I$, where $\om_i$ is a valuation satisfying $\om_i(\phi_i(\al))>0$ according to Lemma \ref{lemma}.
\end{itemize}
\end{thm}

\begin{proof}
Assume that $R_\nu[\al]$ is integrally closed. Since $k_\nu=k_{\nu^h}$ and $\ol{f}=\prod_{i=1}^s \ol{\phi_i^{l_i}}$, Hensel's Lemma yields a factorization $f=\prod_{i=1}^s f_i$ over $K^h$ such that $\ol{f_i}=\ol{\phi_i^{l_i}}$ for $i=1, \dots, s$. In order for us to invoke Lemma \ref{Endler}, we need to show that the factors $f_1, \dots, f_s$ are all irreducible over $K^h$. If $i\in \{1, \dots, s\}- I$, then $f_i$ is immediately irreducible over $K^h$ since $\ol{f_i}=\ol{\phi_i}$ is irreducible. If $i\in I$, we set to show that $f_i$ is $\nu^h$-Eisenstein and thus irreducible, by Lemma \ref{Eis 1}. Since $R_\nu[\al]$ is integrally closed and $l_i\geq 2$, it follows from Lemma \ref{inf} that $\Ga^+$ has a minimum element $\s$ and $\nu^{\mbox{\tiny G}}(r_i)=\s$. Notice that as $\Ga_\nu=\Ga_{\nu^h}$, $\s$ is the minimum element of $\Ga_{\nu^h}^+$ as well. Let $q_i^*, r_i^*\in R_{\nu^h}[X]$ be, respectively, the quotient and remainder upon the Euclidean division of $f_i$ by $\phi_i$. Letting $G_i =\prod_{j=1,\, j\neq i}^s f_j$, we write $f=f_i G_i=q_i^* \phi_i G_i+r_i^* G_i$. Using the Euclidean division again to divide $r_i^*G_i$ by $\phi_i$, let $r_i^*G_i=q_i^{**} \phi_i + r_i^{**}$, with $q_i^{**}, r_i^{**}\in R_{\nu^h}[X]$. Then we have $f=q_i^* \phi_i G_i + q_i^{**}\phi_i+r_i^{**}= (q_i^*G_i + q_i^{**})\phi_i + r_i^{**}$. Due to the uniqueness of the remainder, we get $r_i=r_i^{**}$. Thus, ${\nu^h}^{\mbox{\tiny G}}(r_i^{**})={\nu^h}^{\mbox{\tiny G}}(r_i)=\nu^{\mbox{\tiny G}}(r_i)=\s$. If ${\nu^h}^{\mbox{\tiny G}}(r_i^*)>\s$, then ${\nu^h}^{\mbox{\tiny G}}(r_i^*G_i)>\s$ and so ${\nu^h}^{\mbox{\tiny G}}(r_i^{**})>\s$, a contradiction. Thus, ${\nu^h}^{\mbox{\tiny G}}(r_i^*)=\s$ and we conclude that $f_i$ is $\nu^h$-Eisenstein as desired. It now follows from Lemma \ref{Endler} that there are exactly $s$ valuations $\om_1, \dots, \om_s$ of $L$ extending $\nu$ and follows from Lemma \ref{sigma} that $l_i \om_i(\phi_i(\al))=\s$ for the valuation $\om_i$ of $L$ extending $\nu$ with $\om_i(\phi_i(\al))>0$.

Conversely, assume that there are exactly $s$ valuations $\om_1, \dots, \om_s$ of $L$ extending $\nu$, and if $I=\{i\,|\, l_i\geq 2,\, i=1, \dots, s \}$ is not empty, then $l_i\om_i(\phi_i(\al))$ is the minimum element of $\Ga_\nu^+$ for every $i\in I$ and every $\om_i$ satisfying $\om_i(\phi_i(\al))>0$. If $I=\varnothing$, then $R_\nu[\al]$ is integrally closed by Theorem \ref{main}. Assume that $I\neq \varnothing$. Following Theorem \ref{main}, in order to show that $R_\nu[\al]$ is integrally closed, it suffices to prove that $\nu^{\mbox{\tiny G}}(r_i)=\s$ for every $i\in I$, where $\s=\mbox{min}(\Ga_\nu^+)$. Let $\om_i$ be the valuation of $L$ extending $\nu$ such that $\om_i(\phi_i(\al))>0$ (by Lemma \ref{lemma}). Then, by assumption, $l_i\om_i(\phi_i(\al))=\s$. Write $f$ in the form $f=m_i \phi_i^{l_i} + n_i \phi_i +r_i$ for $m_i, n_i\in R_\nu[X]$ with $\nu^{\mbox{\tiny G}}(m_i)=0$, $\ol{\phi_i}$ does not divide $\ol{m_i}$, $\nu^{\mbox{\tiny G}}(n_i)>0$, and $\mbox{deg}(r_i)<\mbox{deg}(\phi_i)$. Since $f(\al)=0$, we have $r_i=-m_i\phi_i^{l_i}-n_i \phi_i$. We can see (using Lemma \ref{lemma}-(ii)) that $$\om_i(n_i(\al)\phi_i(\al))=\om_i(n_i(\al))+\om_i(\phi_i(\al))>\om_i(n_i(\al))\geq \nu^{\mbox{\tiny G}}(n_i)\geq \s,$$
and (where $\om_i(m_i(\al))=\nu^{\mbox{\tiny G}}(m_i)=0$ by Lemma \ref{lemma}-(iv))
$$\om_i(m_i(\al)\phi_i(\al)^{l_i})=\om_i(\phi_i(\al)^{l_i})=l_i\om_i(\phi_i(\al))=\s.$$ So, $$\om_i(r_i(\al))=\om_i\left (-m_i(\al)\phi_i(\al)^{l_i}-n_i(\al)\phi_i(\al)\right )=\s.$$
Since $\mbox{deg}(r_i)<\mbox{deg}(\phi_i)$, $\ol{\phi_i}$ does not divide $\ol{r_i}$. So, by Lemma \ref{lemma}-(iv), $\nu^{\mbox{\tiny G}}(r_i)=\om_i(r_i(\al))=\s$ and the proof is complete.
\end{proof}

With the notation of Theorem \ref{main 2}, for a valuation $\om_i$ of $L$ extending $\nu$, we denote the ramification index $[\Ga_{\om_i}:\Ga_\nu]$ by $e(\om_i/\nu)$ and the residue degree $[k_{\om_i}:k_\nu]$ by $f(\om_i/\nu)$. The following {\it fundamental inequality} is well-known (see \cite[Theorem 3.3.4]{Eng} for instance):
$$\sum_{i=1}^s e(\om_i/\nu)f(\om_i/\nu) \leq [L:K].$$
When $R_\nu[\al]$ is integrally closed, we calculate in the following corollary the ramification indices $e(\om_i/\nu)$ and residue degrees $f(\om_i/\nu)$ and consequently show that the above inequality is indeed an equality.

\begin{cor}\label{ram-res}
Keep the notation and assumptions of Theorem \ref{main 2}. If $R_\nu[\al]$ is integrally closed, then $e(\om_i/\nu)=l_i$ and $f(\om_i/\nu)=\mbox{deg}(\phi_i)$, for every $i=1, \dots, s$, and furthermore $\sum_{i=1}^s e(\om_i/\nu)f(\om_i/\nu) = [L:K]$.
\end{cor}

\begin{proof}
We show first that $e(\om_i/\nu)\geq l_i$ and $f(\om_i/\nu)\geq \mbox{deg}(\phi_i)$ for every $i=1, \dots, s$. If $l_i=1$ for some $i=1, \dots, s$, then clearly $e(\om_i/\nu)\geq l_i$. Since $\ol{f_i}=\ol{\phi_i}$, it follows that for any root $\al_i$ of $f_i$, $\ol{\phi_i}$ is the minimal polynomial of $\ol{\al_i}$ over $k_\nu$ and so
$$\mbox{deg}(\phi_i)=\mbox{deg}(\ol{\phi_i})=[k_\nu(\ol{\al_i}):k_\nu]\leq [k_{\om_i}:k_\nu]=f(\om_i/\nu).$$
If $l_i \geq 2$ for some $i=1, \dots, s$, then it follows from Theorem \ref{main 2} (ii) that $\om_i(\phi_i(\al))=\s/l_i$, where $\s=\min(\Ga^+_\nu)$. So, $\Ga_\nu \subseteq \Ga[\s/l_i] \subseteq \Ga_{\om_i}$ and
$$l_i =[\Ga_\nu[\s/l_i]:\Ga_\nu]\leq [\Ga_{\om_i}:\Ga_\nu] =e(\om_i/\nu).$$
Also, for a root $\al_i$ of $f_i$, we have $\ol{\phi_i}(\ol{\al_i})^{l_i}=\ol{f_i}(\ol{\al_i})= 0$ implying that $\ol{\phi_i}(\ol{\al_i})=0$ in $k_{\om_i}$. Since $\ol{\phi_i}$ is monic and irreducible over $k_\nu$, we have
$$\mbox{deg}(\phi_i)=\mbox{deg}(\ol{\phi_i})=[k_\nu(\ol{\al_i}):k_\nu]\leq [k_{\om_i}:k_\nu]=f(\om_i/\nu).$$
Now, by the above argument we get the inequality
$$\sum_{i=1}^s e(\om_i/\nu) f(\om_i/\nu) \geq \sum_{i=1}^s l_i \mbox{deg}(\phi_i) = \sum_{i=1}^s  l_i \mbox{deg}(\ol{\phi_i}) = \mbox{deg}(\ol{f}) = \mbox{deg}(f) =[L:K].$$
Thus, by this inequality and the fundamental inequality, we get the claimed equality
$$\sum_{i=1}^s e(\om_i/\nu) f(\om_i/\nu)=[L:K].$$
Furthermore, since $l_i \leq e(\om_i/\nu)$ and $\mbox{deg}(\phi_i)\leq f(\om_i/\nu)$ for every $i=1, \dots, s$ with \linebreak $\sum_{i=1}^s l_i \, \mbox{deg}(\phi_i)=  \sum_{i=1}^s e(\om_i/\nu) f(\om_i/\nu)$, we conclude that $l_i=e(\om_i/\nu)$ and $\mbox{deg}(\phi_i)=f(\om_i/\nu)$ for every $i=1, \dots, s$.
\end{proof}

\section{{\bf Applications and Examples}}\label{applications}

\begin{cor}
Keep the assumptions of Theorem \ref{main} with $f(X)=X^n-a\in R_\nu[X]$ irreducible of degree $n\geq 2$ and $a\in M_\nu$.

1. If $\Ga_\nu^+$ has no minimum element, then $R_\nu[\al]$ is not integrally closed.

2. If  $\min(\Ga_\nu^+)=\s$, then $R_\nu[\al]$ is integrally closed if and only if $\nu(a)=\s$.
\end{cor}

\begin{proof}
This is a direct applications of Theorem \ref{main}.
\end{proof}

\begin{cor}\label{Eis 2}
Keep the assumptions of Theorem \ref{main}. Let $\min(\Ga_\nu^+)=\s$ and $g\in R_\nu[X]$ monic. If $g$ is $\nu$-Eisenstien and $L=K(\t)$ for some root $\t$ of $g$, then $R_\nu[\t]$ is integrally closed.
\end{cor}

\begin{proof}
By Lemma \ref{Eis 1}, $g$ is irreducible over $K$. Now, the remaining part follows directly from Theorem \ref{main}.
\end{proof}

\begin{cor}
Let $f(X)=X^n-a\in R_\nu[X]$, $\min(\Ga^+)=\s$, $\nu(a)=m\s$ for some $m\in \N$. Let $L=K(\t)$ be a root of $f(X)$. If $m$ and $n$ are coprime, then $f$ is irreducible over $R$ and $R[\t^v/\pi^u]$ is the integral closure of $R$ in $L$, where $\pi\in R_\nu$ is such that $\nu(\pi)=\s$, and $u,v\in \Z$ are the unique integers such that $mv-nu=1$ and $0\leq v < n$.
\end{cor}

\begin{proof}
Let $A=a^v/\pi^{nu}$. Then, $\nu(A)=(mv-nu)\s=\s$. By Lemma \ref{Eis 1}, $g(X)=X^n-A$ is irreducible over $R_\nu$. Furthermore, $\t^v/\pi^u$ is a root of $g$. By Corollary \ref{Eis 2}, $R_\nu[\t^v/\pi^u]$ is integrally closed.
\end{proof}

\begin{example}
Let $\geq$ be the lexicographic order defined on $\Z^2$; that is: $(a,b)\geq (c,d)$ if and only if ($a<c$) or ($a=c$ and $b\leq d$). Then $(\Z^2, \geq)$ is a totally ordered abelian group. Let $F$ be a field and $K=F(X,Y)$, the field of rational functions over $F$ in an indeterminates $X$ and $Y$. Define the valuation $\nu:K\to \Z^2 \cup \{\infty\}$ by $0\neq \sum_{i,j} a_{i,j} X^i Y^j \mapsto \min \{(i,j)\,|\, a_{i,j}\neq 0\}$ for $\sum_{i,j} a_{i,j} X^i Y^j\in F[X,Y]$, $0\mapsto \infty$, and $\nu^{\mbox{\tiny G}}(f/g)=\nu^{\mbox{\tiny G}}(f)-\nu^{\mbox{\tiny G}}(g)$ for $f,g\in F[X,Y]$ with $g\neq 0$. Then, obviously, $\nu$ is a discrete valuation on $K$ of rank 2 whose value group is $\Ga_\nu=(\Z^2, \geq)$. Let $f(Z)=Z^3+aZ+b\in R_\nu[Z]$ be irreducible and $L=K(\al)$ for some root $\al$ of $f$. Assume that $\nu(a)>(0,0)$ and $\nu(b)>(0,0)$. Then, $\ol{f}(Z) = Z^3$. Let $r$ be the remainder upon the Euclidean division of $f$ by $Z$. Noting that $\min(\Ga_\nu^+)=(0,1)  $, it follows from Theorem \ref{main}, $R_\nu[\al]$ is integrally closed if and only if $\nu^{\mbox{\tiny G}}(r)=(0,1)$. In particular, if $f(Z)=Z^3+Y$, then $R_\nu[\al]$ is integrally closed; while if $f(Z)=Z^3 +YZ+X$, then $R_\nu[\al]$ is not integrally closed.
\end{example}

\begin{example}
Let $(F,\nu)$ be a valued field and $K=F(X)$ the field of rational functions over $F$ in an indeterminate $X$. For some positive irrational real number $\la$, define the valuation \linebreak $\om:K\to \R\cup\{\infty\}$ as follows: $\om(0)=\infty$, for $0\neq f(X)=\sum_{i=0}^n a_i X^i \in F[X]$, set $\displaystyle{\om(f)=\min\{\nu(a_i)+i\la, i\}}$, and for $f, g\in F[X]$ with $g\neq 0$, $\om(f/g)=\om(f)-\om(g)$ (see \cite[Theorem 2.2.1]{Eng}). Let $f(Z)=Z^3+aZ+b\in R_\om[Z]$ be irreducible and $L=K(\al)$ for some root $\al$ of $f$. If $(F,\nu)$ is the trivial valued field, then $\Ga_\om=\la\Z$. So, in this case, if $\nu(a)>0$ and $\nu(b)>0$, then $\ol{f}(Z)=Z^3$. Hence, by Theorem \ref{main}, $R_\om[\al]$ is integrally closed if and only if $\nu(a)=\la$. In particular, if $f(Z)=Z^3+X$, then $R_\om[\al]$ is integrally closed. If $F=\mathbb{Q}$ and $\nu$ is the $p$-adic valuation on $\mathbb{Q}$ for some prime integer $p$, then $\Ga_\om=\Z+\la \Z$, which is dense in $\R$ and, thus, $\inf(\Ga_\om^+)=0$. So, according to Theorem \ref{main}, $R_\om[\al]$ is integrally closed if and only if $\ol{f}$ is square-free. In particular, if $\nu(a)>0$ and $\nu(b)>0$, then $\ol{f}(Z)=Z^3$ and, thus, $R_\om[\al]$ is not integrally closed.
\end{example}

\section*{Acknowledgement}
A. Deajim would like to express his gratitude to King Khalid University for providing administrative and technical support. He would also like to thank the University Council and the Scientific Council of King Khalid University for approving a sabbatical leave request for the academic year 2018-2019, during which this paper was prepared and submitted.

\end{document}